\documentclass[a4paper,11pt]{article}
\usepackage[ngerman, english]{babel}
\usepackage[T1]{fontenc}
\usepackage[utf8]{inputenc}
\usepackage{amsmath}
\usepackage{amssymb}
\usepackage{amsthm}
\usepackage{graphicx}
\usepackage{here}
\usepackage{enumerate}
\usepackage{lmodern}
\usepackage{fancyvrb}
\usepackage[plainpages=false]{hyperref}
\usepackage{caption}
\newtheorem{defi}{Definition}
\newtheorem{theo}[defi]{Theorem}
\newtheorem{lem}[defi]{Lemma}

\newtheorem{obs}[defi]{Observation}
\newtheorem{cor}[defi]{Corollary}
\theoremstyle{remark}

\title{Rotation $r$-graphs}

\author{Eckhard Steffen, Isaak H. Wolf \thanks{Funded by Deutsche Forschungsgemeinschaft (DFG) - 445863039} \\
		Paderborn University, Department of Mathematics, \\ Warburger Str. 100, 33098 Paderborn,
		Germany. \\ es@upb.de, isaak.wolf@upb.de}

\date{}

\begin{document}

\maketitle

\begin{abstract}
We study rotation $r$-graphs and show that for every $r$-graph $G$ of odd regularity there is a simple rotation $r$-graph $G'$ such that $G$ can be obtained form $G'$ by a finite number of $2$-cut reductions. As a consequence, some hard conjectures as the (generalized) Berge-Fulkerson Conjecture and Tutte's 3- and 5-flow conjecture can be reduced to rotation $r$-graphs.
\end{abstract}

\section{Introduction and basic definitions}
We consider finite graphs that may have parallel edges but no loops. A graph without parallel edges is called simple. A tree is homeomorphically irreducible if it has no vertex of degree 2 and if a graph $G$ has a homeomorphically irreducible spanning tree $T$, then $T$ is called a hist and $G$ a hist graph. The study of hist graphs has been a very active area of research within graph theory for decades, see for example \cite{Albertson_Thomassen_HIST_intro_1990, Harary_enumerate_HIST_1959, Ito_etal_HIST_2022}. 

Cubic hist graphs then have a spanning tree in which every vertex has either degree 1 or 3. They further have the nice property that their edge set can be partitioned into the edges of the hist and of an induced cycle on the leaves of the hist. A snark is a bridgeless cubic graph that is not 3-edge-colorable. 
Informally, a rotation snark is a snark that has a balanced hist and a $\frac{2 \pi}{3}$-rotation symmetry which fixes one vertex. 
Hoffmann-Ostenhof and Jatschka \cite{hoffmannostenhof2017special} studied rotation snarks and conjectured that there are infinitely many non-trivial rotation snarks. This conjecture was proved by M\'{a}\v{c}ajov\'{a} and \v{S}koviera \cite{Skoviera_Superpos_Rot_Snarks_2021} by constructing an infinite family of cyclically 5-edge-connected rotation snarks. 
It is natural, to ask whether some notoriously difficult conjectures can be proved for rotation snarks. As a first result in this direction, Liu et al.~\cite{CQ_Rot_Snarks_BFC_2021} proved that the Fulkerson-Conjecture
\cite{fulkerson1971blocking} is true for the rotation snarks of  \cite{Skoviera_Superpos_Rot_Snarks_2021}. 

We generalize the notion of rotation snarks to $r$-graphs of odd regularity and show that every $r$-graph of odd regularity can be "blown up" to a simple rotation $r$-graph (which produces many small edge-cuts). As a consequence, some hard long-standing open conjectures can be reduced to simple rotation $r$-graphs. However, our proof heavily relies on the fact that we allow $2$-edge-cuts. It would be interesting to study rotation $r$-graphs with high edge-connectivity.

\subsection{Rotation $r$-graphs}

The degree of a vertex $v$ of a graph $G$ is denoted by $d_G(v)$. The set of neighbors of a set $S \subseteq V(G)$ is 
$N_G(S) = \{u \colon u \in V(G) \setminus S \text{ and } u \text{ is adjacent to a vertex of } S\}$. If $S$ consists of a single vertex $v$, then we write $N_G(v)$ instead of $N_G(\{v\})$. The subscript may be omitted if there is no harm of confusion. A cycle is a graph whose components are eulerian.

	An \textit{automorphism} of a graph $G$ is a mapping $\alpha \colon V(G) \to V(G)$, such that for every two vertices $u,v\in V(G)$ the number of edges between $u$ and $v$ is the same as the number of edges between $\alpha(u)$ and $\alpha(v)$. For every $v \in V(G)$, the smallest positive integer $k$ such that $\alpha^k(v)=v$ is denoted by $d_{\alpha}(v)$.

	Let $v$ be a vertex of a tree $T$. An automorphism $\alpha$ of $T$ is \textit{rotational} with respect to $v$, if $d_{\alpha}(v)=1$ and $d_{\alpha}(u)=d_G(v)$ for every $u \in V(T) \setminus \{v\}$. The unique tree with vertex degrees in $\{1,r\}$ and a vertex $v$ with distance $i$ to every leaf is denoted by $T_i^r$. Vertex $v$ is unique and it is called the \textit{root} of $T_i^r$.

	An $r$-regular graph $G$ is an \textit{$r$-graph}, if $\vert \partial(S) \vert \geq r$ for every $S \subseteq V(G)$ of odd cardinality, where $\partial(S)$ is the set of edges with precisely one end in $S$. An $r$-regular graph $G$ is a \textit{$T_i^r$-graph}, if $G$ has a spanning tree $T$ isomorphic to $T_i^r$. If, additionally, $G$ has an automorphism that is rotational on $T$ (with respect to the root), then $G$ is a \textit{rotation} $T_i^r$-graph. Note that $G$ can be embedded in the plane (crossings allowed) such that the embedding has a $\frac{2 \pi}{r}$-rotation symmetry fixing the root. A \textit{rotation $r$-graph} is an $r$-graph that is a rotation $T_i^r$-graph for some integer $i$.  

\begin{obs}
	Let $r,i$ be positive integers, let $G$ be a $T_i^r$-graph with corresponding spanning tree $T$ and let $L$ be the set of leaves of $T$. The order of $G$ is $1+ \sum_{j=0}^{i-1} r(r-1)^{j}$, which is even if and only if $r$ is odd. In particular, if $G$ is an $r$-graph, then $r$ is odd, $G[L]$ is a cycle and $E(G)$ can be partitioned into $E(T)$ and $E(G[L])$.
\end{obs}

\subsection{Main result} \label{Subsec:Main}

Let $G$ be an $r$-graph and $S \subseteq V(G)$ be of even cardinality. If $\vert \partial(S) \vert =2$, then $N_G(S)$ consists of precisely two vertices, say $u,v$. Let $G'$ be obtained from $G$ by deleting $G[S] \cup \partial(S)$ and adding the edge $uv$. We say that $G'$ is obtained from $G$ by a \textit{$2$-cut reduction} (of $S$). The following theorem is the main result of this paper.

\begin{theo} \label{main result 1}
	Let $r$ be a positive odd integer. For every $r$-graph $G$ there is a simple rotation $r$-graph $G'$, such that $G$ can be obtained from $G'$ by a finite number of $2$-cut reductions.
\end{theo}

	The following corollary is a direct consequence of Theorem \ref{main result 1}.

\begin{cor}
	Let $r$ be a positive odd integer and let $A$ be a graph-property that is preserved under $2$-cut reduction. Every $r$-graph has property $A$ if and only if every simple rotation $r$-graph has property $A$.
\end{cor}

	As a consequence, some notoriously difficult conjectures can be reduced to rotation $r$-graphs.

\begin{cor} \label{Cor: BF}
	Let $r$ be an odd integer. The following statements are equivalent:
	\begin{enumerate}
		\item (generalized Fulkerson Conjecture \cite{seymour1979multi}) 
		For $r \geq 1$, every $r$-graph $G$ has a collection of $2r$ perfect matchings such that every edge of $G$ is in precisely two of them. 
		\item For $r \geq 1$, every simple rotation $r$-graph $G$ has a collection of $2r$ perfect matchings such that every edge of $G$ is in precisely two of them.
		\item (generalized Berge Conjecture) For $r \geq 1$, every $r$-graph $G$ has a collection of $2r-1$ perfect matchings such that every edge of $G$ is in at least one of them. 
		\item For $r \geq 1$, every simple rotation $r$-graph $G$ has a collection of $2r-1$ perfect matchings such that every edge of $G$ is in at least one of them. 
	\end{enumerate}
\end{cor}

\begin{proof}
	Let $G$ and $G'$ be two $r$-graphs such that $G$ can be obtained from $G'$ by a 2-cut reduction of a set $S \subset V(G')$. For parity reasons, every perfect matching of $G'$ contains either both or no edges of $\partial(S)$. Hence, each perfect matching of $G'$ can be transformed into a perfect matching of $G$, which implies the equivalences ($1 \Leftrightarrow 2$) and ($3 \Leftrightarrow 4$). The equivalence ($1 \Leftrightarrow 3$)	is proved in \cite{Mazzuoccolo_Equiv_gen_BFC_2013}.
\end{proof}

	For $r=3$, statement 3 of Corollary \ref{Cor: BF} is usually attributed to Berge and statement 1 was first put in print in \cite{fulkerson1971blocking}. Fan and Raspaud \cite{fan1994fulkerson} conjectured that every $3$-graph has three perfect matchings such that every edge is in at most two of them. Equivalent formulations of this conjecture are studied in \cite{Jin_et_al_Fano_flows_Fan_Raspaud}.

\begin{cor} \label{Cor: FR} Let $r$ be an odd integer and $ 2 \leq k \leq r-1$. Every $r$-graph has $r$ perfect matchings, such that each edge is in at most $k$ of them if and only if every simple rotation $r$-graph has $r$ perfect matchings, such that each edge is in at most $k$ of them.
\end{cor}

	A \textit{nowhere-zero $k$-flow} of a graph $G$ is a mapping $f:E(G) \to \{\pm 1,\dots,\pm (k-1)\}$ together with an orientation of the edges, such that the sum of $f$ over all incoming edges of $v$ equals the sum of $f$ over all out going edges of $v$ for every $v \in V(G)$. In 1954, Tutte \cite{tutte_1954}] stated his seminal conjecture that every bridgeless graph admits a nowhere-zero 5-flow. It is folklore that the conjecture can be reduced to snarks. In 1972, Tutte formulated the no less challenging conjecture that every simple 5-graph admits a nowhere-zero 3-flow (see \cite{bondy1976graph} unsolved problem 48). Admitting a nowhere-zero $k$-flow is invariant under $2$-cut reduction. Hence, we obtain the following consequences of Theorem \ref{main result 1}.

\begin{cor} \label{Cor: 5-flow}
	Every snark admits a nowhere-zero $5$-flow if and only if every simple rotation snark admits a nowhere-zero $5$-flow.
\end{cor}

\begin{cor} \label{Cor: 3-flow}
	Every $5$-graph admits a nowhere-zero $3$-flow if and only if every simple rotation $5$-graph admits a nowhere-zero $3$-flow.
\end{cor}

\section{Proof of Theorem \ref{main result 1}}

\subsection{Preliminaries}

For the proof of Theorem \ref{main result 1} we will use the following lemma. The non-trivial direction of the statement is proved by Rizzi in \cite{rizzi1999indecomposable} (Lemma 2.3).

\begin{lem} [\cite{rizzi1999indecomposable}] \label{Rizzilemma}
	Let $G$ be an $r$-regular graph and let $S \subseteq V(G)$ be a set of odd cardinality with $\vert \partial(S) \vert=r$. Let $G_{S}$ and $G_{\bar{S}}$ be the graphs obtained from $G$ by contracting $S$ and $\bar{S}=V(G) \setminus S$ to single vertices and removing all resulting loops, respectively. The graph $G$ is an $r$-graph, if and only if $G_S$ and $G_{\bar{S}}$ are both $r$-graphs.
\end{lem}

	Let $G$ be an $r$-graph and $T$ be a spanning tree of $G$. We need the following two expansions of $G$ and $T$. 

{\bf Edge-expansion:}
Let $e$ be an edge with $e=uv \in E(G) \setminus E(T)$. Let $G'$ be the graph obtained from $G-e$ by adding two new vertices $u', v'$ that are connected by $r-1$ edges, and adding two edges $uu'$ and $vv'$. Extend $T$ to a spanning tree $T'$ of $G'$ by adding the edges $uu'$ and $vv'$ (see Figure \ref{fig:1}). For $S = \{u,u',v'\}$ it follows with Lemma \ref{Rizzilemma} that $G'$ is an $r$-graph. 

\begin{figure}[H]
	\centering
	\includegraphics[height=1.7cm]{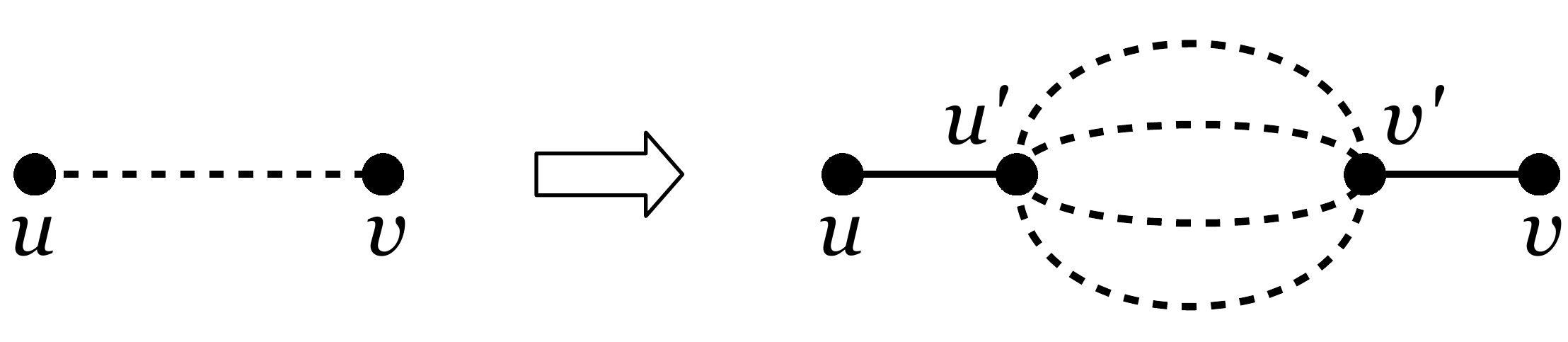}
	\caption{An edge-expansion in the case $r=5$. The solid edges belong to the spanning tree $T'$.}
	\label{fig:1}
\end{figure}

{\bf Leaf-expansion:} Let $r$ be odd. Let $l$ be a leaf of $T$ and let $u$ be the neighbor of $l$ in $T$. Let $K$ be a copy of the complete graph on $r$ vertices and $V(K) = \{l_1, \dots, l_r\}$. Let $G'$ be the $r$-regular graph obtained from $G-l$ and $K$ by connecting every vertex of $K$ with a neighbor of $l$. Without loss of generality we assume $ul_1 \in E(G')$. Extend $T-l$ to a spanning tree $T'$ of $G'$ by adding $V(K)$ and the edges $ul_1$ and $l_1l_j$ for $j \in \{2, \dots,r\}$. Vertex $l_1$ has degree $r$ in $T'$, whereas all other vertices of $K$ are leaves of $T'$. Furthermore, if $l$ has distance $d$ to a vertex $x \in V(T)$, then the $r-1$ leaves $l_2, \dots, l_r$ of $T'$ have distance $d+1$ to $x$ in $T'$.

	Since $K_{r+1}$ is an $r$-graph, $G'$ is an $r$-graph by Lemma \ref{Rizzilemma}. We note that a leaf-expansion of leaf $l$ has the following properties:
\begin{itemize}
\item[(i)] In $G'$, no vertex of $K$ is incident with parallel edges.
\item[(ii)] Let $S\subseteq V(G)$ be a set of even cardinality with $l \in S$ and $\vert \partial(S) \vert =2$. In the leaf-expansion $G'$, the set $S'=S \setminus \{l\} \cup V(K)$ is of even cardinality and satisfies $\vert \partial(S') \vert =2$. Moreover, the graph obtained from $G$ by a 2-cut reduction of $S$ is the same graph that is obtained from $G'$ by a 2-cut reduction of $S'$.
\end{itemize}

	An example of leaf-expansion is shown in Figure \ref{fig:2}. 

\begin{figure}[H]
	\centering
	\includegraphics[height=4cm]{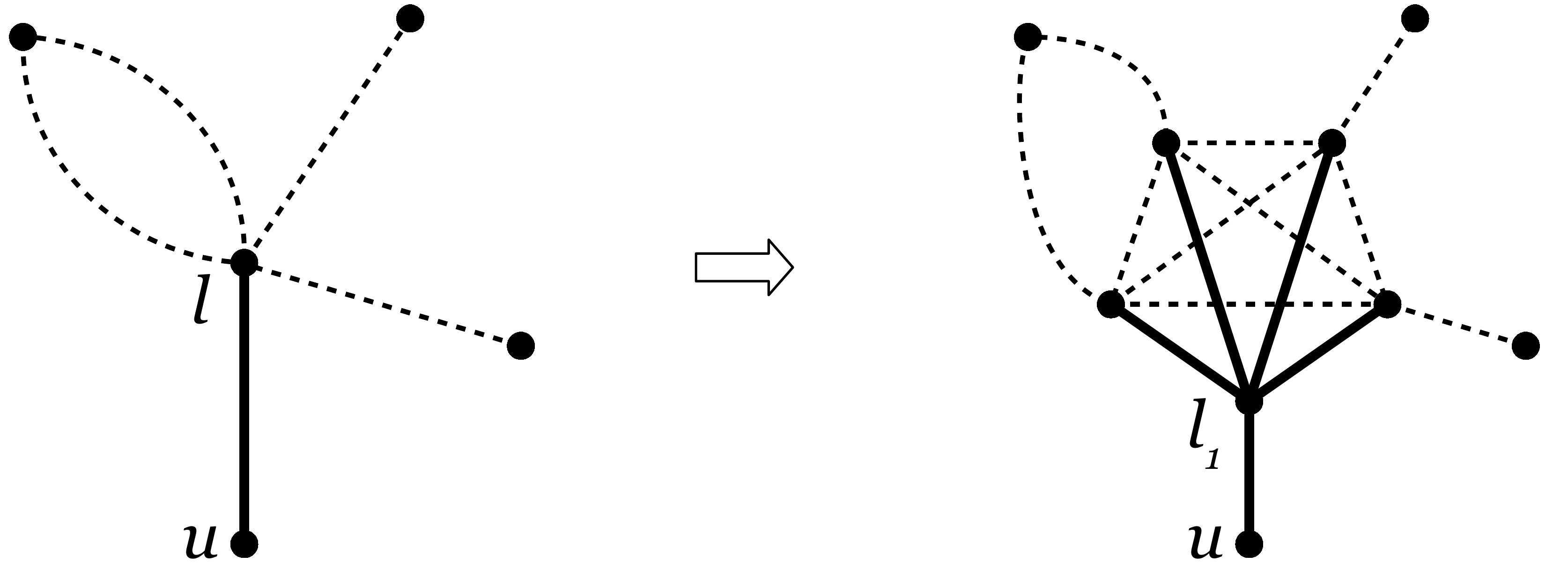}
	\caption{An example of a leaf-expansion of the leaf $l \in V(G)$ in the case $r=5$. The solid edges belong to the spanning trees $T$ and $T'$ respectively.}
	\label{fig:2}
\end{figure}

\subsection{Construction of $G'$}
 
Let $r \geq 1$ be an odd integer and $G$ be an $r$-graph. We will construct $G'$ in two steps.\\

	1. We construct a simple $r$-graph $H$ with a spanning tree $T_H$ isomorphic to $T_{i}^r$ for some integer $i$ such that $G$ can be obtained from $H$ by a finite number of $2$-cut reductions.

	Let $T_G$ be an arbitrary spanning tree of $G$. Apply an edge-expansion on every edge of $E(G) \setminus E(T_G)$ to obtain an $r$-graph $H_1$ with spanning tree $T_1$. Clearly, $G$ can be obtained from $H_1$ by 2-cut reductions. Furthermore, $V(G) \subseteq V(H_1)$, every vertex of $V(G)$ has degree $r$ in $T_1$ and all vertices of $V(H_1) \setminus V(G)$ are leaves of $T_1$.

	Let $x \in V(H_1)$ with $d_{T_1}(x) = r$ and let $d$ be the maximal distance of $x$ to a leaf in $T_1$. Repeatedly apply leaf-expansions until every leaf has distance $d+1$ to $x$. Let $H_2$ be the resulting graph and $T_2$ be the resulting spanning tree of $H_2$. By the construction, $T_2$ is isomorphic to $T_{d+1}^r$, where $x$ is the root of $T_2$. By the definition of $d$, we applied a leaf-expansion of $l$ for every leaf $l$ of $T_1$. Hence, the graph $H_2$ is simple by property (i) of leaf-expansions. Furthermore, no expansion of a vertex in $V(G)$ (and degree $r$ in $T_1$) is applied. As a consequence, property (ii) of leaf-expansions implies that $G$ can be obtained from $H_2$ by 2-cut reductions. Thus, by setting $H=H_2$ and $T_H=T_2$ we obtain a graph with the desired properties.\\

	2. We construct a simple rotation $r$-graph $G'$ from which $H$ can be obtained by a 2-cut reduction.

	Let $y_1,\dots,y_r$ be the neighbors of $x$ in $H$. Let $R$ be an arbitrary simple rotation $r$-graph with a spanning tree $T_R$ isomorphic to $T_{d+1}^r$. For example, such a graph can be obtained from the rotational $T_1^r$-graph $K_{r+1}$ by repeatedly applying leaf-expansions. Let $x_R$ be the root of $T_R$ and let $\alpha_R$ be the corresponding rotational automorphism. Label the neighbors of $x_R$ with $z_1,\dots,z_{r}$ such that $\alpha_R(z_i)=z_{i+1}$ for every $i \in \{1,\dots,r\}$, where the indices are added modulo~$r$. 

	Take $r$ copies $H^1,\dots,H^r$ of $H$ and $(r-1)^2-r$ copies $R^1,\dots,R^{(r-1)^2-r}$ of $R$. In each copy we label the vertices accordingly by using an upper index. For example, if $v$ is a vertex of $H$, then $v^i$ is the corresponding vertex in $H^i$. Furthermore, the automorphism of $R^i$ that correspond to $\alpha_R$ will be denoted by $\alpha_{R^i}$. Delete the root in each of the $(r-1)^2$ copies, i.e.~in each copy of $H$ and in each copy of $R$. The resulting $r(r-1)^2$ vertices of degree $r-1$ are called root-neighbors. 

	Take a tree $T$ isomorphic to $T_2^r$ with root $x_T$. The graph $T \setminus x_T$ consists of $r$ pairwise isomorphic components, thus it has a rotation automorphism $\alpha_T$ with respect to $x_T$. Let $l_1,\dots,l_{r-1}$ be the leaves of one component of $T \setminus x_T$. Clearly, the set of leaves of $T$ is given by ${\{{\alpha^i_T}(l_j) \mid i \in \{0,\dots,r-1\}, j \in \{1,\dots,r-1\}\}}$, where $\alpha^0_T=id_T$. 

	Connect the $r(r-1)$ leaves of $T$ with the $r(r-1)^2$ root-neighbors by adding $r(r-1)^2$ new edges as follows. For every $i \in \{1,\dots,r\}$ define an ordered list $N_i$ of root-neighbors and an ordered list $L_i$ of leaves of $T$ by
\begin{align*}
N_i:=(y_1^i,\dots,y_r^i, z_i^1,\dots,z_i^{(r-1)^2-r})
\quad \text{and } \quad
L_i:=({\alpha^{i-1}_T}(l_1),\dots,{\alpha^{i-1}_T}(l_{r-1})).
\end{align*}

	The list $N_i$ has $(r-1)^2$ entries, whereas $L_i$ has $r-1$ entries. For each $i \in \{1,\dots,r\}$, connect the first $r-1$ entries of $N_i$ with the first entry of $L_i$ by $r-1$ new edges; connect the second $r-1$ entries of $N_i$ with the second entry of $L_i$ by $r-1$ new edges and so on. The set of new edges is denoted by $E$ and the resulting graph by $G'$. In Figure \ref{fig:3} the construction of $G'$ is shown in the case $r=3$.

\begin{figure}[H]
	\centering
	\includegraphics[height=7cm]{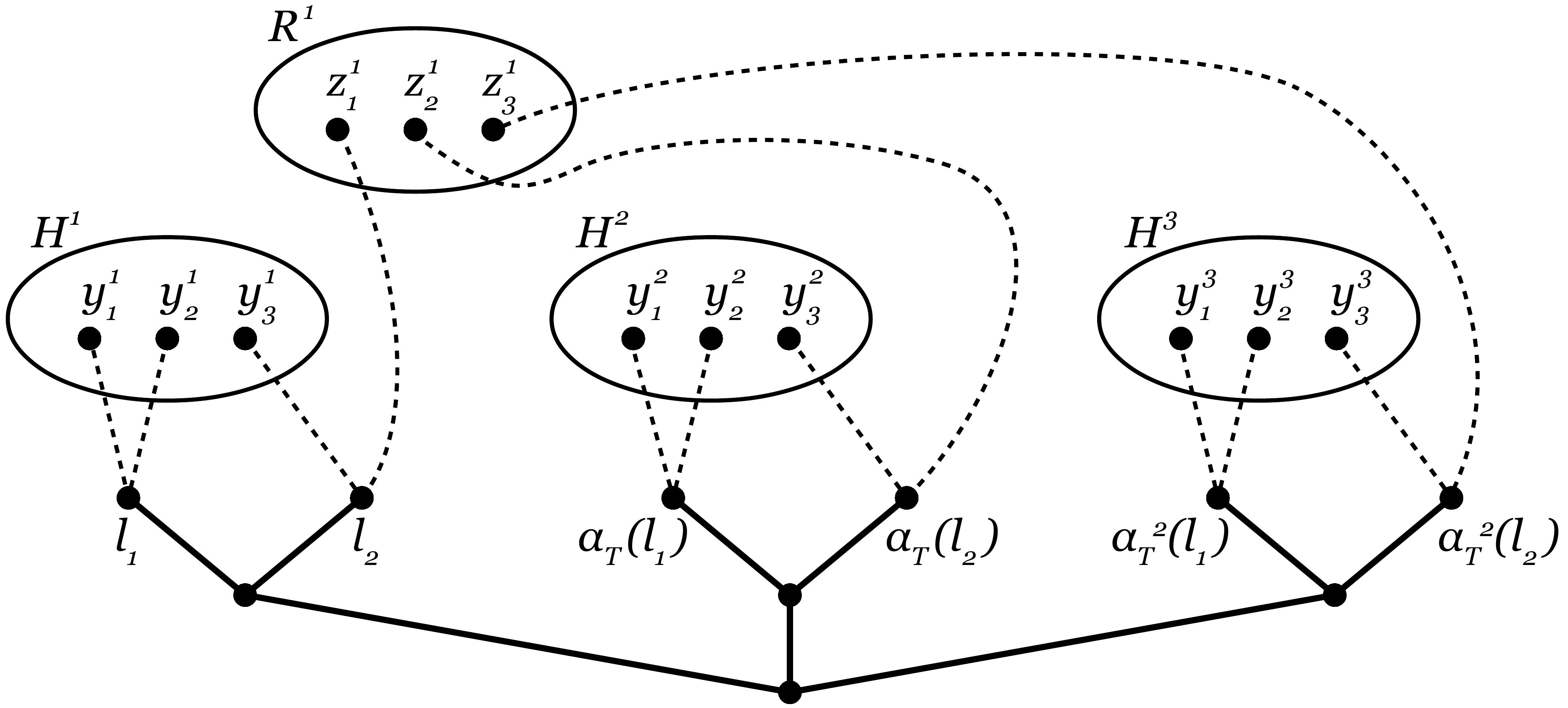}
	\caption{The construction of $G'$ in the case $r=3$. The solid edges belong to $T$; the dashed edges belong to $E$.}
	\label{fig:3}
\end{figure}

	Every root-neighbor appears exactly once in the lists $N_1,\dots,N_r$, whereas every leaf of $T$ appears exactly once in the lists $L_1,\dots,L_r$. Consequently, $G'$ is an $r$-regular simple graph with a spanning tree $T_{G'}$ that is obtained from the union of the trees of each copy of $H$ and $R$ (without its roots) and $T$ by adding the edge set $E$. Note that $T_{G'}$ is isomorphic to $T_{d+3}^r$ and $x_T$ is the root of $T_{G'}$. Let $\alpha_{G'} \colon V(G') \to V(G')$ be defined as follows:
 \begin{align*}
\alpha_{G'}(v)= \begin{cases} 
 				\alpha_T(v) & \text{ if } v \in V(T), \\
 				\alpha_{R^i}(v) & \text{ if } v \in V(R^i) \setminus \{x_{R}^i\},~ i \in \{1, \dots, (r-1)^2-r\}, \\
 				v^{i+1} & \text{ if } v=v^i \in V(H^i) \setminus \{x^i\},~ i \in \{1, \dots, r\} \text{ and the indices are added modulo $r$}.
 			\end{cases}	
 \end{align*}
	By definition, $\alpha_{G'}$ is an automorphism of $G' \setminus E$ and $T_{G'} \setminus E$ that fixes the root $x_T$ of $T$ and satisfies $d_{\alpha_{G'}}(v)=r$ for every other vertex $v$ of $G'$. For $i \in \{1,...,r\}$, if we apply $\alpha_{G'}$ on each element of $N_i$ (or $L_i$ respectively), then we obtain the ordered list $N_{i+1}$ (or $L_{i+1}$ respectively), where the indices are added modulo $r$. As a consequence, if $uv \in E$, then $\alpha_{G'}(u) \alpha_{G'}(v) \in E$ and hence, $\alpha_{G'}$ is an automorphism of $G'$ and a rotational automorphism of $T_{G'}$.

	To see that $G'$ is an $r$-graph, transform $G'$ as follows: for each $i \in \{1,\dots,r\}$ contract all vertices in $V(H^i) \setminus x^i$ to a vertex $\bar{H}^i$ and for every $j \in \{1,\dots,(r-1)^2-r\}$ all vertices in $V(R^j) \setminus x_R^j$ to a vertex $\bar{R}^j$, and remove all loops that are created (see Figure~\ref{fig:4}). The resulting graph is an $r$-regular bipartite graph and therefore, an $r$-graph. Since every copy of $H$ and of $R$ is an $r$-graph, it follows by successively application of Lemma \ref{Rizzilemma} that $G'$ is an $r$-graph.

\begin{figure}[H]
	\centering
	\includegraphics[height=6cm]{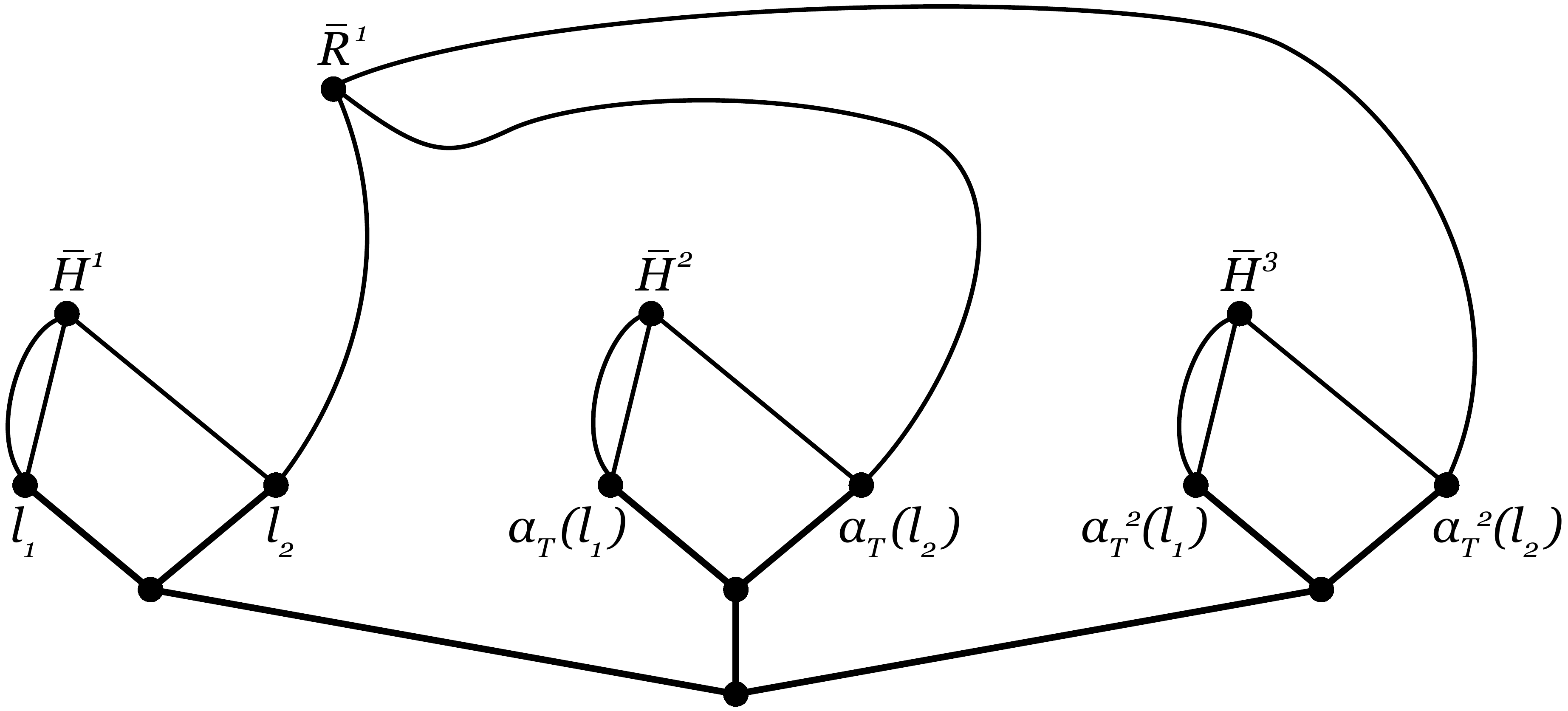}
	\caption{The graph constructed from $G'$ in the case $r=3$.}
	\label{fig:4}
\end{figure}

	Last, the set $S \subseteq V(G')$ defined by $S=V(H^1) \setminus\{x^1\}\cup \{l_1\}$ is a set of even cardinality, such that $\vert \partial(S) \vert =2$. Applying a 2-cut reduction on $V(G') \setminus S$ transforms $G'$ into the copy $H^1$ of $H$. In conclusion, $G$ can be obtained from $G'$ by a finite number of 2-cut reductions, which completes the proof.

\section{Concluding remarks}
The graph $G'$ constructed in the proof of Theorem \ref{main result 1} has many small edge-cuts. It would be interesting to construct and study highly edge-connected rotation $r$-graphs. For example, is there an $r$-edge-connected rotation $r$-graph with chromatic index $r+1$ for every positive odd integer $r$? It might also be possible to prove some of the conjectures mentioned in Corollaries \ref{Cor: BF} - \ref{Cor: 3-flow} for some families of rotation $r$-graphs with high edge-connectivity.


\begin{thebibliography}{10}
	
	\bibitem{Albertson_Thomassen_HIST_intro_1990}
	M.~O. Albertson, D.~M. Berman, J.~P. Hutchinson, and C.~Thomassen.
	\newblock Graphs with homeomorphically irreducible spanning trees.
	\newblock {\em J. Graph Theory}, 14(2):247--258, 1990.
	
	\bibitem{bondy1976graph}
	J.~A. Bondy, U.~S.~R. Murty, et~al.
	\newblock {\em Graph theory with applications}, volume 290.
	\newblock Macmillan London, 1976.
	
	\bibitem{fan1994fulkerson}
	G.~Fan and A.~Raspaud.
	\newblock {F}ulkerson's conjecture and circuit covers.
	\newblock {\em Journal of Combinatorial Theory, Series B}, 61(1):133--138,
	1994.
	
	\bibitem{fulkerson1971blocking}
	D.~R. Fulkerson.
	\newblock Blocking and anti-blocking pairs of polyhedra.
	\newblock {\em Mathematical programming}, 1(1):168--194, 1971.
	
	\bibitem{Harary_enumerate_HIST_1959}
	F.~Harary and G.~Prins.
	\newblock The number of homeomorphically irreducible trees, and other species.
	\newblock {\em Acta Math.}, 101:141--162, 1959.
	
	\bibitem{hoffmannostenhof2017special}
	A.~Hoffmann-Ostenhof and T.~Jatschka.
	\newblock Special hist-snarks.
	\newblock {\em arXiv:1710.05663}, 2017.
	
	\bibitem{Ito_etal_HIST_2022}
	T.~Ito and S.~Tsuchiya.
	\newblock Degree sum conditions for the existence of homeomorphically
	irreducible spanning trees.
	\newblock {\em J. Graph Theory}, 99(1):162--170, 2022.
	
	\bibitem{Jin_et_al_Fano_flows_Fan_Raspaud}
	L.~Jin, G.~Mazzuoccolo, and E.~Steffen.
	\newblock Cores, joins, and the {F}ano-flow conjectures.
	\newblock {\em Disc. Mathematicae Graph Theory}, 38(1):165--175, 2018.
	
	\bibitem{CQ_Rot_Snarks_BFC_2021}
	S.~Liu, R.-X. Hao, and C.-Q. Zhang.
	\newblock Rotation snark, {B}erge-{F}ulkerson conjecture and {C}atlin's 4-flow
	reduction.
	\newblock {\em Appl. Math. Comput.}, 410:Paper No. 126441, 9, 2021.
	
	\bibitem{Skoviera_Superpos_Rot_Snarks_2021}
	E.~M\'{a}\v{c}ajov\'{a} and M.~\v{S}koviera.
	\newblock Superposition of snarks revisited.
	\newblock {\em European J. Combin.}, 91:Paper No. 103220, 21, 2021.
	
	\bibitem{Mazzuoccolo_Equiv_gen_BFC_2013}
	G.~Mazzuoccolo.
	\newblock An upper bound for the excessive index of an {$r$}-graph.
	\newblock {\em J. Graph Theory}, 73(4):377--385, 2013.
	
	\bibitem{rizzi1999indecomposable}
	R.~Rizzi.
	\newblock Indecomposable $r$-graphs and some other counterexamples.
	\newblock {\em J. Graph Theory}, 32(1):1--15, 1999.
	
	\bibitem{seymour1979multi}
	P.~D. Seymour.
	\newblock On multi-colourings of cubic graphs, and conjectures of {F}ulkerson
	and {T}utte.
	\newblock {\em Proceedings of the London Mathematical Society}, 3(3):423--460,
	1979.
	
	\bibitem{tutte_1954}
	W.~T. Tutte.
	\newblock A contribution to the theory of chromatic polynomials.
	\newblock {\em Canadian Journal of Mathematics}, 6:80–91, 1954.
	
\end{thebibliography}


\end{document}